\documentclass{article}
\usepackage[utf8]{inputenc}
\usepackage{amsmath}
\usepackage{amsfonts}
\usepackage{amssymb}
\usepackage{amsthm}
\usepackage{geometry}
\usepackage{enumerate}
\usepackage{bbm}
\usepackage{url}
\usepackage[usenames,dvipsnames]{color}
\usepackage{algorithm, algorithmic}

\usepackage[colorlinks=false,,linktoc=allbookmarksopen=true,linktocpage,pdftex]{hyperref}
\usepackage[pdftex]{bookmark}

\newcommand\Bigpar[1]{\Bigl(#1\Bigr)}

\def\rompar(#1){\textup(#1\textup)}    

\def\xexp(#1){e^{#1}}
\newcommand\ceil[1]{\lceil#1\rceil}

\newcommand{\refT}[1]{Theorem~\ref{#1}}

\newcommand{\refCl}[1]{Claim~\ref{#1}}

\newcommand{\indic}[1]{\mathbbm{1}_{\{{#1}\}}}

\newcommand{\cPr}{\mathbb{P}}
\newcommand{\E}{\mathbb{E}}
\newcommand{\Bin}{\operatorname{Bin}}
\newcommand{\Var}{\operatorname{Var}}

\newtheorem{theorem}{Theorem}[section]

\newtheorem{remark}[theorem]{Remark}

\newtheorem{conjecture}[theorem]{Conjecture}
\newtheorem{claim}[theorem]{Claim}

\title{Rainbow cycles for families of matchings}
\author{Ron Aharoni\thanks{Faculty of Mathematics, Technion, Haifa~32000, Israel and MIPT. E-mail: {\tt ra@technion.ac.il}. Research supported by the Israel Science Foundation (ISF) grant no. 2023464 and the Discount Bank Chair at the Technion, and the European Union's Horizon 2020 research and innovation programme under the Marie Skldowska-Curie grant agreement no.\ 823748.}
~and He Guo\thanks{Faculty of Mathematics, Technion, Haifa~32000, Israel. E-mail: {\tt hguo@campus.technion.ac.il}.}}
\date{October 27, 2021; Revised on November 10, 2023}

\begin{document}

\maketitle
\begin{abstract}
    Given a graph~$G$ and a coloring of its edges, a subgraph of~$G$ is called {\em rainbow} if its edges have distinct colors. The {\em rainbow girth} of an edge coloring of~$G$ is the minimum length of a rainbow cycle in~$G$. 
    A  generalization of the famous Caccetta-H\"aggkvist conjecture, proposed by the first author, is that if in an coloring of the edge set of an $n$-vertex graph by $n$ colors, in which each color class is of size~$k$,  the rainbow girth is at most~$\lceil \frac{n}{k} \rceil$. In the known examples for sharpness of this conjecture the color classes are stars, suggesting that when the color classes are matchings, the result may be  improved. We show that  the rainbow girth of~$n$ matchings of size at least 2 is~$O(\log n)$.
\end{abstract}

\section{Introduction}
The {\em girth} $g(G)$ of a graph $G$ is the minimal length of a cycle in it. Given a graph~$G$ and  a (not necessarily proper) coloring of~its edges, a subgraph~of~$G$ is called \emph{rainbow} if its  edges have distinct colors. The \emph{rainbow girth} $rg(G)$  of~$G$ (actually, of its edge-coloring) is the minimum length of a rainbow cycle. All the above definitions apply to both the directed and undirected cases, where in the directed case the cycles are assumed to be directed. 

The  famous Caccetta-H\"{a}ggkvist conjecture \cite{CaccettaHaggkvist} (below - CHC) is that any digraph $G$ on $n$ vertices satisfies $g(G) \le \lceil \frac{n}{\delta^+(G)}\rceil$, where $\delta^+(G)$ is the minimal out-degree of a vertex. There has been constant progress~\cite{ChvatalSzemeredi, Hamidoune, HoangReed, Nishimura} on the problem. In particular it has been shown that 

(a) The CHC is true if $n\ge 2\delta^+(G)^2-3\delta^+(G)+1$ ~\cite{Shen1},  and

(b)  $g(G)\le n/\delta^+(G)+73$ for all~$G$ ~\cite{Shen2}.

In~\cite{ADH2019} a possible generalization of CHC was suggested. 

\begin{conjecture}\label{rainbow}
Let $G$ be an undirected  $n$-vertex graph. For any  edge coloring of~$E(G)$ with~$n$ colors such that each color class has size at least~$k$, we have $rg(G) \le \ceil{n/k}$. 
\end{conjecture}

Devos et. al.~\cite{DDFGGHMM2021}  proved this conjecture for $r=2$. In~\cite{ABCGZ2021} a stronger version of the conjecture was proved when all sets are of size $1$ or $2$. 

For a directed edge $e=uv$  let $n(e)$
be the undirected pair $\{u,v\}$. To see that Conjecture~\ref{rainbow} is a generalization of CHC, given a directed graph $G$, for every vertex $u$ let $S(u)=\{n(uv) \mid uv \in E(G)\}$  be the star  of edges leaving $u$,  with their direction removed. We claim that an undirected rainbow cycle $v_1v_2\ldots v_k$ for the sets $S(v)$ gives rise to a directed cycle in $G$. Otherwise there exists a vertex $v_i$ such that $\{v_i,v_{i+1}\}=n(v_iv_{i+1})$ and $\{v_i,v_{i-1}\}=n(v_{i}v_{i-1})$. But this   contradicts the fact that only one edge is chosen from $S(v_i)$ for participation in the rainbow cycle. Thus the
CHC is  equivalent to the case of Conjecture \ref{rainbow} in which
 the color classes are $n$ stars, each centered at a different vertex. Hence the sharpness of CHC implies that of Conjecture \ref{rainbow}. 
 
The standard example showing that the former is the case is the graph on~$\{1,2,\dots, n\}$
with  edges $\{i,i+1\},\{i, i+2\},\dots,\{i, i+k\}$ for $i=1,2,\dots, n$  (indices taken  modulo $n$). But this example is not unique. 
Bondy~\cite{bondy} noted that if $G$ and $H$ are two graphs witnessing the sharpness of CHC, then blowing each vertex of $G$ by a copy of $H$ yields another such example (the blow-up is called also the {\em lexicographic product} of $G$ and $H$).  In \cite{razborov} more examples are given (accompanied by a conjecture that these  exhaust all possible cases of equality in the conjecture). 

By the above argument, every example witnessing the sharpness of CHC gives rise to an example witnessing the sharpness of Conjecture \ref{rainbow}. In fact, all known extreme
examples for this conjecture are obtained  this way,  in particular they have stars as the sets of edges.  
This suggests that in the antipodal case to that of stars, when the sets of edges are matchings, the conjecture can be strengthened. (``Antipodal'' is with respect to the  covering number, which is $1$ in a star, and the number of edges in a matching.)
Indeed, a simple observation is that for the first open case  of CHC, that of $\delta^+=\frac{n}{3}$ (in which a directed triangle is conjectured to exist) the rainbow undirected version is trivial when the sets of edges are matchings. In this case, it is enough that the arithmetic mean of the sizes of the sets is larger than $\frac{n}{4}$, because then by Mantel's theorem there exists a triangle contained in the union of the sets, and if the sets are matchings then a triangle is necessarily rainbow.

\section{Rainbow cycles for matchings}

Our main result is a corroboration of the intuition that sets of  matchings have small rainbow girth. 

\begin{theorem}\label{thm:main}
There exists a constant $C$ such that for any $n$-vertex graph~$G$ and edge coloring of~$G$ with~$n$ colors, if each color class is a matching of size 2, then the rainbow girth of~$G$ is at most~$C\log n$.
\end{theorem}
\begin{remark}
The assumption  
that $G$ is a graph and not a multigraph breeds no loss of generality, 
since a double edge is a rainbow digon, meaning that the rainbow girth is $2$. \end{remark}

A key ingredient in the proof is a result by Bollob\'as and Szemer\'edi~\cite{BS2002} on the girth of sparse graphs.

\begin{theorem}\label{thm:girthsparsegraph}
For $n\ge 4$ and $k\ge 2$, every $n$-vertex graph with $n+k$ edges has girth at most
\[\frac{2(n+k)}{3 k}(\log k+\log\log k+4). \]
\end{theorem}

The logarithms are to the base $2$.~\refT{thm:main} will follow from this result, and the following:
 
 \begin{theorem}\label{thm:bridge_to_bs}
 There exist universal $c,\delta>0$,  such that for any large enough $n$, given an $n$-vertex graph~$G$ and an edge coloring of~$G$ with~$n$ colors such that each color class is a matching of size $2$, there exists  a 
 subset~$S$ of $V(G)$ of size at most $cn$  containing the edges of a rainbow set of edges of size at least $(c+\delta)n$.
 \end{theorem}
 
   Note that the last condition entails $c+\delta \le 1$. Once this is proved,  Theorem 
    \ref{thm:main}  follows by applying Theorem \ref{thm:girthsparsegraph} with $k=\delta n$.
 
 The idea of proving~\refT{thm:bridge_to_bs} is that we take a random subset~$S$ of~$V(G)$ and consider the induced subgraph~$G[S]$. The crux of the argument is that the expected number of vertices $\E |S|$ is polynomially in~$n$ less than the expected number of colors of the edges in~$G[S]$. Furthermore,  these two random numbers are concentrated around their expectations, which follows from two well-known concentration inequalities.

\begin{theorem}[Chernoff]\label{thm:chernoff}
Let $X$ be a binomial random variable $\Bin(n,p)$. For any $0<\epsilon<1$, we have
\[\cPr(X\ge (1+\epsilon)\E X)\le \exp(-\epsilon^2\E X/3). \]
\end{theorem}

\begin{theorem}[Chebyshev's inequality]\label{thm:chebyshev}
Let $X$ be a random variable. For any $\epsilon>0$, we have

\[\cPr(|X -\E X|\ge \epsilon \E X)
\le \Var X/(\epsilon\E X)^2, \]
where $\Var X$ is the variance of $X$.
\end{theorem}

\begin{proof}[Proof of \refT{thm:bridge_to_bs}]
Denote the $i$-th color class (which, by our assumption, consists of two disjoint edges) by $M_i$. 
Our assumption that $G$ is a graph and not a multigraph implies that the matchings $M_i$ are disjoint.

A vertex $v$ of $G$ is called \emph{heavy} if there are at least $\epsilon^2 n/10^6$ rainbow edges incident to it, where $\epsilon>0$ is a small constant to be determined later.
Let $D$ be the set of heavy vertices of $G$.
Then we have
\begin{equation}\label{eq:ubonD}
|D|\le \frac{2\cdot 2\cdot n}{(\epsilon^2 n/10^6)}\le \frac{10^7}{\epsilon^2}.
\end{equation}

Let $S=D\cup Z$ be a random vertex subset of $V(G)$, where each vertex of $V(G)\setminus D$ is included in $Z$ independently with probability $p$, for some constant $p$ to be determined later.
Then 
\begin{equation}\label{eq:E|Z|}
    \E |Z|=(n-|D|)\cdot p \quad \text{ and }\quad \E |S|=\E |Z|+|D|\sim np.
\end{equation}
For $1\le i\le n$ let $X_i$ be the indicator random variable that an edge of color~$i$ is contained in $S$, i.e., $X_i:=\indic{\text{an edge of color $i$ is contained in $S$}}$. Since each vertex is included with probability at least $p$ and $X_i$ is an increasing event with respect to the probability that a vertex is included in~$S$, by inclusion-exclusion we have
\[ \E X_i \ge 2p^2-p^4. \] 
Let 
\begin{equation}\label{eq:def:X}
    X:=\sum_{i=1}^n X_i.
\end{equation}
We have 
\begin{equation}\label{eq:EX}
 \E X\ge n(2p^2-p^4).
\end{equation}
By~\refT{thm:chernoff}, for fixed ~$0<p<1$ and $\epsilon>0$, when~$n$ is large enough we have
\begin{equation}\label{eq:probof|S|}
   \cPr\Bigpar{|S|\ge (1+\epsilon)np}\le \cPr\Big(|Z|\ge (1+\epsilon/2)\E |Z|\Big)\le  \exp(-\Omega(n)). 
\end{equation}
So,  with  probability tending to 1 as $n$ tends to infinity,
\begin{align}
    |S|&\le (1+\epsilon)np. \label{eq:|S|ub}
    \end{align}

Writing $p- (2p^2-p^4)=p(p-1)(p^2+p-1)$,
we see that for $\frac{-1+\sqrt{5}}{2}<p<1$, we have $p< 2p^2-p^4$, yielding the separation between $\E X$ and $\E |S|$, needed for the application of \refT{thm:girthsparsegraph}.

\begin{claim}\label{claim:lbX}
There exist constants $p\in (\frac{-1+\sqrt{5}}{2},1)$ and $\epsilon>0$ such that 
\begin{equation}\label{eq:def:eps}
    (1-\epsilon)(2p^2-p^4)-(1+\epsilon)p \ge [(2p^2-p^4)-p]/3>0,
\end{equation}
and with  probability at least 0.9 for all large $n$,    
    \begin{align}
    X&\ge (1-\epsilon)n(2p^2-p^4).\label{eq:Xlb}
\end{align}
\end{claim}

We fix $0.618\approx\frac{-1+\sqrt{5}}{2}< p<1$ and $\epsilon(p)>0$ satisfying~\eqref{eq:def:eps}.

To prove~\eqref{eq:Xlb}, we shall apply Chebyshev's inequality. For this purpose we have to estimate $\Var X$. With a look at~\eqref{eq:def:X}, we have  
\begin{align}
    \Var X = \E X^2-(\E X)^2= \sum_{i,j} (\E X_iX_j- \E X_i\E X_j).
\end{align}

Note that if the edges in the color classes $i,j$ are vertex-disjoint, then $X_i$ and $X_j$ are independent and $\E X_iX_j- \E X_i\E X_j =0$.

Since the matchings $M_j$ are disjoint, for every $i\in [n]$  at most $6=\binom{4}{2}$  matchings $M_j$ can have an edge contained in  $\bigcup M_i$. This means that there exist at most $2\cdot 6n$ pairs $(M_i, M_j)$ such that an edge from $M_j$ is contained in $\bigcup M_i$, or vice versa. Thus the contribution of such pairs to $\Var X$ is at most $O(n)$.

Apart from vertex-disjointness, there are  two more possible forms of $ M_i \cup  M_j$:

I. three connected components: one 2-path and two disjoint edges, or 

II.  two vertex-disjoint 2-paths.

Examine Case I. 
Let $M_i=\{a,b\}$, where $a=xy,~b=uv$, and let $M_j=\{c,d\}$, where $c=xz$ and $d=st$.

If $x$ is a heavy vertex, then $\E X_iX_j-\E X_i\E X_j=0$: since $\{a\subseteq S\}=\{y\in S\}$, $\{b\subseteq S\}$, $\{c\subseteq S\}=\{z\in S\}$, and $\{d\subseteq S\}$ are mutually independent, 
\begin{align*}
&\E X_i \E X_j \\
=&\Big(\cPr(a\subseteq S \text{ and }b\not\subseteq S)+\cPr(a\not\subseteq S \text{ and }b\subseteq S)+\cPr(a\subseteq S \text{ and }b\subseteq S)\Big)\\
&\cdot\Big(\cPr(c\subseteq S \text{ and }d\not\subseteq S)+\cPr(c\not\subseteq S \text{ and }d\subseteq S)+\cPr(c\subseteq S \text{ and }d\subseteq S)\Big)\\
=&\Big(\cPr(a\subseteq S)\cPr(b\not\subseteq S)+\cPr(a\not\subseteq S)\cPr(b\subseteq S)+\cPr(a\subseteq S)\cPr(b\subseteq S)\Big)\\
&\cdot \Big(\cPr(c\subseteq S)\cPr(d\not\subseteq S)+\cPr(c\not\subseteq S)\cPr(d\subseteq S)+\cPr(c\subseteq S)\cPr(d\subseteq S)\Big)\\
=&\cPr(\text{$a,c \subseteq S$ and $b,d \not \subseteq S$})+\cPr(\text{$a,d \subseteq S$ and  $b,c \not \subseteq S$})+\cPr(\text{$b,c \subseteq S$ and $a,d \not \subseteq S$})\\
&+\cPr(\text{$b,d \subseteq S$ and $a,c \not \subseteq S$})
+\cPr(\text{$a,b,c\subseteq S$ and $d\not\subseteq S$})+\cPr(\text{$a,b,d\subseteq S$ and $c\not\subseteq S$})\\
&+\cPr(\text{$a,c,d\subseteq S$ and $b\not\subseteq S$})+\cPr(\text{$b,c,d\subseteq S$ and $a\not\subseteq S$})+\cPr(\text{$a,b,c,d\subseteq S$})\\
=&\E X_iX_j.
\end{align*}

If $x$ is not a heavy vertex, the contribution to $\Var X$ for such $(M_i,M_j)$ is at most $4\cdot n\cdot \frac{\epsilon^2 n}{10^6}$, since there are at most $n$ ways to choose $M_i$, four ways to choose $x$, and at most $\frac{\epsilon^2 n}{10^6}$ ways to choose $M_j$.

In case II, let $M_i=\{a,b\}$, where $a=xy,~b=uv$, and let $M_j=\{c,d\}$, where $c=xz$ and $d=ut$. 

If both $x$ and $u$ are heavy vertices, then similarly to case I, we have $\E X_iX_j-\E X_i\E X_j=0$.  Otherwise the contribution to $\Var X$ for such $(M_i,M_j)$ is at most $4\cdot n\cdot \frac{\epsilon^2 n}{10^6}$.

Summing, we have 
\[ \Var X=O(n)+2\cdot4\cdot n\cdot \frac{\epsilon^2 n}{10^6}\le \frac{\epsilon^2 n^2}{10^5}.   \]
Since $1/2\le p<1$, we have $2p^2-p^4\ge 7/16$.
Applying Chebyshev's inequality, we have
\begin{align*}
    \cPr(X \le (1-\epsilon) n\cdot (2p^2-p^4))&\le \cPr(X\le \E X -\epsilon n \cdot (2p^2-p^4)) \\
    &\le \frac{\Var X}{(\epsilon n \cdot (2p^2-p^4))^2} \le \frac{\epsilon^2 n^2}{10^5 \epsilon^2 n^2(7/16)^2}\le 1/10.
\end{align*}
This proves~\refCl{claim:lbX}.

 Combining~\eqref{eq:Xlb} with~\eqref{eq:probof|S|}, we have that with probability at least 1/2 for all large $n$,
 \[ |S|\le (1+\epsilon)np< (1-\epsilon)n(2p^2-p^4) \le X. \]

Theorem \ref{thm:bridge_to_bs} now follows from~\eqref{eq:def:eps}, upon taking $c:=(1+\epsilon)p$ and $\delta:= [(2p^2-p^4)-p]/3$.
\end{proof}

\subsection{Fewer than $n$ sets}
In the original CHC, 
each set of edges is a star associated with a vertex, hence it was natural that there are~$n$ sets. In the rainbow undirected case there is no natural choice of the number of sets. Indeed, the main theorem is valid also with fewer than $n$ sets.

\begin{theorem}\label{thm:alpha}
For any constant $\alpha> \frac{3\sqrt{6}}{8}$, there exists a constant $C$ such that for any $n$-vertex graph~$G$ and edge coloring of~$G$ with~$\alpha n$ colors, if each color class is a matching of size 2, then the rainbow girth of~$G$ is at most~$C\log n$.
\end{theorem}

 
 
\begin{proof}
For the argument in the proof of \refT{thm:bridge_to_bs} to work with $\alpha n$ colors, we need to have 
 $p< \alpha (2p^2-p^4)$ for some $p$ (which would imply separation between $\E |S|\sim pn$ and $\E X\ge \alpha n(2p^2-p^4)$).

  Thus we need to find a minimal  $0<\alpha_0<1$ such that $p=\alpha_0 (2p^2-p^4)$ for some $p\in(0,1)$. This will happen when the two curves $y(p)=p$ and $y(p)=\alpha_0 (2p^2-p^4)$ are tangent, namely $(\alpha_0 (2p^2-p^4))'=p'=1$. The above two constraints and $\alpha_0p\neq 0$ imply that $\alpha_0=\frac{3\sqrt{6}}{8}$ and the only feasible $p$ is $\frac{\sqrt{6}}{3}$.
\end{proof}

It would be interesting to find the optimal value of~$\alpha$ in~\refT{thm:alpha}. For example, it should be at least~$1/2$: for even~$n$ and $G=C_n$, assume the edges of the cycle in order are~$e_1,e_2,\dots,e_n$. If we color the edges $e_i$ and~$e_{i+n/2}$ by color~$i$, then there is no rainbow cycle.

\noindent\textbf{Acknowledgements.}
We thank the referees for their helpful suggestions. We also thank Xiaozheng Chen, Yuhui Cheng, and Ruonan Li for identifying omissions in earlier versions of this work.

\small

\begin{thebibliography}{10}

\bibitem{ABCGZ2021}
R.~Aharoni, E.~Berger, M.~Chudnovski, H.~Guo, and S.~Zerbib.
\newblock Non-uniform degrees and rainbow versions of the {C}acceta-{H}\"aggkivst conjecture.
\newblock {\em SIAM J. Discrete Math.} {\bf 37} (2023), 1704--1714.


\bibitem{ADH2019}
R.~Aharoni, M.~DeVos, and R.~Holzman.
\newblock Rainbow triangles and the {C}accetta-{H}\"aggkvist conjecture.
\newblock {\em J. Graph Theory} {\bf 92} (2019), 347--360.

\bibitem{BS2002}
B.~Bollobás and E.~Szemerédi.
\newblock Girth of sparse graphs.
\newblock {\em J. Graph Theory} {\bf 39} (2002), 194--200.


\bibitem{bondy}
J.A.~Bondy. 
\newblock Counting subgraphs: A new approach to the Caccetta-H\"aggkvist conjecture. 
\newblock {\em Discrete Math.} {\bf 165/166} (1997), 71--80.


\bibitem{CaccettaHaggkvist}
L.~Caccetta and R.~H\"{a}ggkvist.
\newblock On minimal digraphs with given girth.
\newblock {\em Congress. Numer.} {\bf 21} (1978), 181--187.



\bibitem{ChvatalSzemeredi}
V.~Chv\'atal and E.~Szemer\'edi.
\newblock Short cycles in directed graphs.
\newblock {\em J. Combin. Theory Ser. B} {\bf 35} (1983), 323--327.

\bibitem{DDFGGHMM2021}
M.~DeVos, M.~Drescher, D.~Funk, S.~González Hermosillo de~la Maza, K.~Guo,
  T.~Huynh, B.~Mohar, and A.~Montejano.
\newblock Short rainbow cycles in graphs and matroids.
\newblock {\em J. Graph Theory} {\bf 96} (2021), 192--202.


\bibitem{Hamidoune}
Y.O.~Hamidoune.
\newblock A note on minimal directed graphs with given girth.
\newblock {\em J. Combin. Theory Ser. B} {\bf 43} (1987), 343--348.

\bibitem{HoangReed}
C.T.~Ho\`ang and B.~Reed.
\newblock A note on short cycles in diagraphs.
\newblock {\em Discrete Math.} {\bf 66} (1987), 103--107.

\bibitem{Nishimura}
T.~Nishimura.
\newblock Short Cycles in Digraphs.
\newblock {\em Discrete Math.} {\bf 38} (1988), 295--298.




\bibitem{razborov}
A.A.~Razborov.
\newblock On the Caccetta-H\"aggkvist Conjecture with Forbidden Subgraphs.
\newblock{\em J. Graph Theory} {\bf 74} (2013), 236--248.







\bibitem{Shen1}
J.~Shen.
\newblock On the girth of digraphs.
\newblock{\em Discrete Math.} {\bf 211} (2000), 167--181.


\bibitem{Shen2}
J.~Shen.
\newblock On the Caccetta-H\"aggkvist Conjecture.
\newblock{\em 	Graphs Combin.} {\bf 18} (2002), 645--654.



\end{thebibliography}

\normalsize
\end{document}